\documentclass[11pt]{article}
\usepackage[a4paper,margin=2cm]{geometry}
\usepackage[english]{babel}
\usepackage{latexsym,amsmath,enumerate,graphics,enumerate,amsthm,tikz,hyperref,float}
\usepackage[mathscr]{euscript}
\usepackage[affil-it]{authblk}
\usepackage{enumerate,amsthm,dsfont,pstricks}
\usepackage{latexsym,amsmath,amssymb,amscd,wrapfig,graphicx}

\newtheorem{theorem}{Theorem}[section]
\newtheorem{lemma}[theorem]{Lemma}
\newtheorem{proposition}[theorem]{Proposition}
\newtheorem{corollary}[theorem]{Corollary}

\theoremstyle{definition}
\newtheorem{definition}[theorem]{Definition}

\newtheorem{remark}[theorem]{Remark}
\theoremstyle{remark}

\let\phi=\varphi
\def\epsilon{\varepsilon}

\def\0{\mathbf{0}}

\newcommand{\comment}[1]{}

\numberwithin{equation}{section}

\let\epsilon=\varepsilon

\makeatletter
\def\@maketitle{%
  \newpage
  \null
  \vskip 2em%
  \begin{center}%
  \let \footnote \thanks
    {\Large\bfseries \@title \par}%
    \vskip 1.5em%
    {\normalsize
      \lineskip .5em%
      \begin{tabular}[t]{c}%
        \@author
      \end{tabular}\par}%
    \vskip 1em%
    {\normalsize \@date}%
  \end{center}%
  \par
  \vskip 1.5em}
\makeatother

\begin{document}

\title{\sc \huge On the linearity of order-isomorphisms}

\author{Bas Lemmens%
\thanks{Email: \texttt{B.Lemmens@kent.ac.uk}}}
\affil{School of Mathematics, Statistics \& Actuarial Science,
University of Kent, Canterbury, CT2 7NX, United Kingdom}

\author{Onno van Gaans%
\thanks{Email: \texttt{vangaans@math.leidenuniv.nl}}}
\affil{Mathematical Institute, Leiden University, P.O.Box 9512, 2300 RA Leiden, The Netherlands}

\author{Hendrik van Imhoff%
\thanks{Email: \texttt{hvanimhoff@gmail.com}}}
\affil{Mathematical Institute, Leiden University, P.O.Box 9512, 2300 RA Leiden, The Netherlands}

\maketitle
\date{}

\begin{abstract}
A  basic problem in the theory of partially ordered vector spaces is to characterise those cones on which  every order-isomorphism is linear. We show that this is the case for every Archimedean cone that equals the inf-sup hull of the sum of its engaged extreme rays. This condition is milder than existing ones and is satisfied by, for example, the cone of positive operators in the space of bounded self-adjoint operators on a Hilbert space. We also give a general form of order-isomorphisms on the inf-sup hull of the sum of all extreme rays of the cone, which extends results of Artstein-Avidan and Slomka to infinite dimensional  partially ordered vector spaces, and  prove the linearity of homogeneous order-isomorphisms  in a variety of new settings.  
\end{abstract}




\section{Introduction}
A fundamental problem in the study of partially ordered vector spaces is to understand the structure of their order-isomorphisms, i.e., order preserving bijections whose inverses are also order preserving. In particular one would like to characterise those  partially ordered vector spaces on which all order-isomorphisms are affine. 

Pioneering research on this problem was motivated by special relativity theory where the causal order is considered on the Minkowski spacetime.
During the 1950s and 1960s several results were obtained in  finite dimensional spaces by  Alexandrov and Ov\v{c}innikova \cite{AO} and Zeeman \cite{Z}, who showed that the order-isomorphisms from the causal cone onto itself are linear. Later Alexandrov \cite{A} extended his result to order-isomorphisms on finite dimensional ordered vector spaces, where every extreme ray of the cone is engaged,  that is to say,  each  extreme ray of the cone lies in the linear span of the other extreme rays. Rothaus \cite{R} obtained a similar result where the domain of the order-isomorphism could also be the interior of the cone, but  he  assumes that the  cone does not have any isolated extreme rays, which is a stronger assumption  than the one used by  Alexandrov. In the 1970s Noll and Sch\"{a}ffer made numerous contributions to this area in a series of papers, \cite{NS,NS2,S1,S2}.
Like Alexandrov, they considered the case where the cone is the sum of its engaged extreme rays, but they do not  require the  partially ordered vector spaces to be finite dimensional.
More recently,  Artstein-Avidan and Slomka \cite{AAS}  obtained a complete description of the order-isomorphisms between finite dimensional partially ordered vector spaces. 

In many natural infinite dimensional  settings the results of Noll and Sch\"{a}ffer are not applicable. A case in point is the space $B(H)_{\mathrm{sa}}$ consisting of bounded self-adjoint operators on a Hilbert $H$, ordered by the cone of positive (semi-definite) operators.  Even though the cone  $B(H)_{\mathrm{sa}}^+$ contains many engaged extreme rays, namely the rays through the rank-one projections, it does not satisfy the condition of Noll and Sch\"{a}ffer.  Even so  Moln\'{a}r \cite{M} showed,  by using operator algebra techniques, that every order-isomorphism on  $B(H)_{\mathrm{sa}}^+$ is linear.  In this paper we obtain a generalisation of \cite[Theorem A]{NS2}    by Noll and Sch\"{a}ffer that is sufficiently strong to yield  Moln\'{a}r's result.

Before we outline the main results in the paper, we point out that the domain on which the order-isomorphisms are considered plays a  key role. In the paper we will work on so called upper sets, i.e, sets  which contain all upper bounds of its elements. Such domains include cones, their interiors, and the whole space.  It turns out that without this assumption order-isomorphisms can be more complicated. Indeed,  \v{S}emrl \cite{Sem} gave a complete characterisation of the order-isomorphisms on order intervals of $B(H)_{\mathrm{sa}}$, which include maps that are not affine.

Our  generalisation  of \cite[Theorem A]{NS2}  exploits the fact that infima and suprema in a partially ordered vector space are preserved under  order-isomorphisms. 
Instead of the conditions imposed  by Noll and Sch\"{a}ffer, we assume that  the cone, $C$, is equal to the inf-sup hull of the positive span of its engaged extreme rays, which is much weaker. 
In other words, we require that each $x\in C$ can be written as $x=\inf_{\alpha\in A} (\sup_{\beta\in B} x_{\alpha,\beta})$, where each $x_{\alpha,\beta}$ belongs to 
\[
[0,\infty)_{\mathcal{R}_E} =\{ r_1+\cdots+r_n\colon r_i\in C\mbox{ an engaged extreme vector of $C$ for all $i$}\},
\] 
and $A$ and $B$ arbitrary index sets. The main result can be formulated as follows.  
\begin{theorem}\label{thrm,ext}
Suppose $U\subseteq (X,C)$ and $V\subseteq (Y,K)$  are upper sets in Archimedean partially ordered vector spaces, and $f\colon U\to V$ is an order-isomorphism. If $(X,C)$ is directed and $C$ equals the inf-sup hull of $[0,\infty)_{\mathcal{R}_E}$,  then $f$ is affine. 
\end{theorem}
A key step in our argument is  Theorem \ref{thrm,ns}, which says that every order-isomorphism  $f$ from $[a,\infty) =\{a+x\colon x\in C\}$ onto $[b,\infty)=\{b+y\colon y\in K\}$  is the restriction of an affine map on the affine span of $[a,\infty)_{\mathcal{R}_E} = a+[0,\infty)_{\mathcal{R}_E} $. 
The proof requires a  careful reworking of some of the ideas in \cite{NS2}.

Of course not every order-isomorphism is affine, Simply consider the space $C(K)$, consisting  of continuous real functions on a compact Hausdorff space $K$, and the map $f\mapsto f^3$. 
On $C(K)$ Sch\"{a}ffer \cite{S1} showed that each order-isomorphism, which is homogeneous (of degree one), is linear. In \cite{S2} he  strengthened this result to general order unit spaces.
In finite dimensional spaces the existence of a disengaged extreme ray in the cone is necessary and sufficient to yield a nonlinear order-isomorphism. This follows from \cite[Theorem 1.7]{AAS}  by Artstein-Avidan and Slomka, who showed that any order-isomorphism in a finite dimensional space has a particular diagonal form. In Section 5 we obtain an infinite dimensional analogue  of this result. We also give an alternative condition that guarantees that all homogeneous order-isomorphisms are linear, which can applied in partially ordered vector spaces without an order unit such as $\ell^p(\mathbb{N})$ spaces.

\section{Preliminaries}
Let $X$ be a real vector space and $C$ be a cone in $X$, so $C$ is convex, $\lambda C\subseteq C$ for all $\lambda \geq 0$, and $C\cap -C=\{0\}$. 
The cone $C$ induces a partial order on $X$ by $x\leq_C y$ if $y-x\in C$. The pair $(X,C)$ is called a {\em partially ordered vector space}. For simplicity we write $\leq$ instead of $\leq_C$ if  $C$ is clear from the context, and we write $x<y$ if $x\leq y$ and $x\neq y$.

A partially ordered vector space $(X,C)$ is said to be {\em Archimedean} if for each $x\in X$ and $y\in C$ with $nx\leq y$ for all $n\geq 1$ we have that $x\leq 0$. 
A subset $G$ of $X$ is said to be {\em directed} if for each $x,y\in G$ there exists $z\in G$ such that $x\leq z$ and $y\leq z$. It is well known that $X$ is directed if and only if $C$ is generating, i.e., $X= C-C$. Given $x\leq y$ we define the {\em order interval} by $[x,y]=\{z\in X: x\leq z\leq y\}$. We denote the cone with apex $a$ by 
\[
[a,\infty) =\{a+x\colon x\in C\}.
\]

Extreme rays of the cone play an important role in this paper. A vector $e\in X\setminus\{0\}$ is called an {\it extreme vector} if  $0\leq e$, and $0\leq x \leq e$ implies that $x=\lambda e$ for some $\lambda\geq 0$, or, if $e\leq 0$, and $e\leq x\leq 0$ implies $x=\lambda e$ for some $\lambda\geq 0$. For an element $x\in C$ we define the {\it ray through} $x$ as $R_x=\{\lambda x:\lambda\geq 0\}$.
If  $e\in C$ is an extreme vector, $R_e$ is said to be an {\em extreme ray}.
The notion of an extreme ray  coincides with the ray being extreme in the convex sense. Indeed, a ray $R$ in $C$ is extreme if, and only if, for any two rays $R_1$ and $R_2$ in $C$  satisfying $R=\alpha R_1+(1-\alpha)R_2$ for some  $\alpha\in (0,1)$ we have that  $R_1=R_2$, see \cite[Lemma 1.43]{AT}.
Given an extreme ray $R$ we call $z+R$ an  {\em extreme half-line with apex} $z$.
The following elementary property of extremal vectors will be used frequently in the sequel, see \cite[Lemma 1.44]{AT}.
\begin{lemma}\label{lem,cad} In a partially ordered vector space $(X,C)$ any three  extremal vectors in $C$ that generate three distinct extremal rays are linearly independent.\end{lemma}

Another useful basic observation is the following. 
\begin{lemma}\label{lem,arch}  Let $(X,C)$  be Archimedean. If $x,y\in X$ are such that $0\leq y\leq x$, and for each $0\leq \lambda\leq 1$  we have that $y\leq \lambda x$ or $\lambda x\leq y$, then there exists a $\mu\geq 0$ such that $y=\mu x$.\end{lemma}
\begin{proof} 
Let $x,y\in X$ be as in the statement. We may assume without loss of generality that $x$ and $y$ are non-zero.
Now  define $\mu=\sup\{\lambda\geq 0\colon \lambda x\leq y\}$.
By assumption $\mu$ is well-defined and $0\leq \mu\leq 1$.

Note that $\mu x\leq y$. Indeed, for $n\geq 1$ we have that $(\mu -1/n)x\leq y$, so that $n(\mu x -y)\leq x$, which implies that $\mu x\leq y$, as $(X,C)$ is Archimedean. 

To show that $y\leq \mu x$ we distinguish two cases: $0\leq \mu< 1$ and $\mu= 1$.
In the case $0\leq \mu<1$ we have that $y\leq (\mu +1/n)x$ for all $n$ sufficiently large. Thus, $n(y-\mu x)\leq x$, which shows that $y\leq\mu x$, as the space is Archimedean. 
If $\mu =1$, then $x=y$, since $y\leq x$ by assumption, and $x=\mu x\leq y$ as showed before. \end{proof}


Given vector spaces $X$ and $Y$, a map $f\colon X\to Y$ is called \emph{affine} if it is a translation of a linear map, that is, there is $a\in X$ such that $x\mapsto f(x+a)-f(a)$ is linear.  

Let $(X,C)$ and $(Y,K)$ be partially ordered vector spaces. A set $U\subseteq X$ is called an \emph{upper set} if $x\in U$ and $y\geq x$ imply $y\in U$. So, $X$, $C$ and translations thereof are all upper sets in $(X,C)$. Let $U\subseteq X$ be an upper set. A map $f\colon U\to Y$ is called \emph{affine} or \emph{linear} if it is the restriction of an affine map $F\colon \mathrm{aff}(U)\to Y$ or a linear map $F\colon \mathrm{span}(U)\to Y$, respectively. If $C$ is generating then we have $\mathrm{aff}(U)=\mathrm{span}(U)=X$. A map $f\colon U\to Y$ is affine if and only if $f(\lambda_1x_1+\cdots+\lambda_nx_n)=\lambda_1f(x_1)+\cdots+\lambda_nf(x_n)$ for all $x_1,\ldots,x_n\in U$ and $\lambda_1,\ldots,\lambda_n\in\mathbb{R}$ with  $\lambda_1+\cdots+\lambda_n=1$ such that $\lambda_1x_1+\cdots+\lambda_nx_n\in U$. It is a well-known fact that, if the upper set $U$ is convex, then $f\colon U\to Y$ is affine if and only if $f$ is \emph{convex-linear}, that is, for each $x,y\in U$ and $0\leq \lambda\leq 1$ we have that $f(\lambda x+(1-\lambda) y) = \lambda f(x)+(1-\lambda) f(y)$.

An element $u$ in a partially ordered vector space $(X,C)$ is an \emph{order unit} if for all $x\in X$ there exists a $\lambda \geq0$ such that $-\lambda u\leq x\leq\lambda u$. If $C$ is generating, then $u\in C$ is an order unit if and only if for every $x\in C$ there exists $\lambda\ge 0$ with $x\le\lambda u$. If $(X,C)$ is Archimedean and $u\in C$ is an order unit then the formula \[
\|x\|_u:=\inf\{\lambda \geq0\colon -\lambda u\leq x \leq \lambda u\}\]
defines a norm on $X$, called the \emph{order unit norm}. A triple $(X,C,u)$, where $(X,C)$ is an Archimedean partially ordered vector space and $u$ is an order unit in $(X,C)$, is called an \emph{order unit space}. In an order unit space we denote the interior of the cone $C$ with respect to the order unit norm by $C^\circ$. The set $C^\circ$ is an upper set and consists of all order units of $(X,C)$.

\section{Linearity of order-isomorphisms}\label{section3}
In the sequel $(X,C)$ and $(Y,K)$ will be Archimedean partially ordered vector spaces.
Initially we only consider order-isomorphisms $f\colon [a,\infty)\to [b,\infty)$, where $a\in X$ and $b\in Y$.
However, the  main result, Theorem \ref{thrm,ext},  holds for more general domains.

A key role in the analysis of  order-isomorphisms is played by extreme half-lines. This idea has been exploited to analyse order-isomorphisms on finite dimensional partially ordered vector spaces \cite{AAS} as well as in infinite dimensions in \cite{NS2}.  In infinite dimensions, however, the extreme half-lines are not as useful, as there are cones that have none or only very few extreme rays.  
The following order theoretic characterization of extreme half-lines  is due to Noll and Sch\"affer, see \cite[Proposition 1]{NS2}. For completeness we provide a proof. 
\begin{proposition}\label{prop,total} If  $(X,C)$ is Archimedean and $x\in X$, then $H\subseteq [x,\infty)$ is an extreme half-line with apex $x$ if and only if $H$ is maximal among subsets $G\subseteq [x,\infty)$ with $x\in G$ that satisfy:
\begin{enumerate}[(P1)]
\item $G$ is directed.
\item For any $y\in G$ the order interval $[x,y]$ is totally ordered.
\item $G$ contains at least two distinct points.
\end{enumerate}\end{proposition}
\begin{proof} Suppose $H\subseteq X$ is maximal among subsets $G\subseteq [0,\infty)$ that satisfy properties (P1)--(P3).
We first argue that $H$ is contained in a half-line.
Let $y,w\in H$ be given, so $x\leq y,w$.
Due to (P1) there exists a $z\in H$ such that $ y,w\leq z$.
Since $\leq$ is preserved under addition, (P2) guarantees that the order interval $[0,z-x]$ is totally ordered.
Moreover, it contains $y-x$, $w-x$, and $\lambda(z-x)$ for all $0\leq \lambda\leq 1$.
Therefore, by Lemma \ref{lem,arch} there exist $\alpha,\beta\geq 0$ such that $y-x=\alpha(z-x)$ and $w-x=\beta(z-x)$.
This shows that $y$ and $w$ are on the half-line through $z$ with apex $x$.  We conclude that any pair of points in $H$ lie on a half-line with apex $x$, and hence $H$ is contained in a half-line with apex $x$. Let $R$ be a ray in $C$ such that $H\subseteq x+R$. 

By (P3) there exists an $r\in C\setminus\{0\}$ such that $x+r\in H$ and $x+R = \{x+\lambda r\colon \lambda \geq 0\}$. Note that $x+R$ satisfies properties (P1) and (P3).
We now show  that $x+R$ also satisfies (P2).
Consider $y=x+\lambda r$ with $\lambda>0$. Then $[x,y]=[x,x+\lambda r]$ equals the interval $[x,r]$ up to dilation.
We know that $[x,x+r]$ is totally ordered, as $x+r\in H$ and $H$ satisfies property (P2). Hence $[x,y]$ is also totally ordered.
It now follows from the maximality assumption on $H$ that $H=x+R$.

To see that $x+R$ is an extreme half-line, we note that $[0,r]$ is totally ordered, as $[x,x+r]$ is totally ordered. It follows from Lemma \ref{lem,arch} that $r$ is an extreme vector. 
 
Conversely, suppose $H=x+R$ is an extreme half-line. Clearly, $H$ satisfies properties (P1)--(P3).
Suppose $G\supseteq H$ also satisfies (P1)--(P3) and $y\in G$.
Since $G$ is directed, there exists a $z\in G$ with $z\geq y,x+r$.
Moreover, $[x,z]$ is totally ordered by (P2) and, hence, $[0,z-x]$ is totally ordered and $y-x,r\in [0,z-x]$. If $y-x\le r$, then there is a $\mu\ge 0$ such that $y-x=\mu r$, as $r$ is extreme, so that $y=x+\mu r\in H$. Otherwise, we have $r\le y-x$ and for each $0\le\lambda\le 1$ we have $\lambda(y-x)\in [0,z-x]$, so $r\le \lambda(y-x)$ or $\lambda(y-x)\le r$. By Lemma  \ref{lem,arch} it follows that there is a $\sigma\ge 0$ such that $r=\sigma(y-x)$. Then $\sigma\neq 0$ and $y=x+\sigma^{-1}r\in H$.
\end{proof}
We note  that property (P3) is only a necessary condition if $C$ does not have any extreme rays and can be dropped otherwise.

As a direct corollary we obtain the following result.
\begin{corollary}\label{thrm,ehl} If $f\colon [a,\infty)\to [b,\infty)$ is an order-isomorphism, then $f$ maps an extreme half-line with apex $x \in [a,\infty)$ onto an extreme half-line with apex $f(x)\in [b,\infty)$.
\end{corollary}
\begin{proof}
Suppose that $R$ is an extreme ray of $C$. Then $f(x+R)\subseteq [f(x),\infty)$ and satisfies properties (P1)--(P3), as $f$ is an order-isomorphism. So by Proposition \ref{prop,total} we find that $f(x+R) =f(x)+S$, where $S$ is an extreme ray of $K$. 
\end{proof}

Our next step is to show that order-isomorphisms $f\colon [a,\infty)\to [b,\infty)$ possess an  additive property on extreme half-lines, which was proved in \cite[Lemma 1]{NS2}. For the reader's convenience we include the proof.
\begin{lemma}\label{lem,add} Let $R$ and $S$ be distinct extreme rays of $C$ and $f\colon [a,\infty)\to [b,\infty)$ be an order-isomorphism. For each  $x\in [a,\infty)$, $r\in R$ and $s\in S$ we have that
\begin{equation} \label{eq:par} 
f(x+r+s)-f(x+s)=f(x+r)-f(x).
\end{equation}
\end{lemma}
\begin{proof} The equality in the statement holds trivially if either $r$ or $s$ equals zero.
Assume $r\neq 0$ and $s\neq 0$.
Then $R_j=x+js+R$ for $j\in\{0,1,2\}$ are three distinct parallel extreme half-lines.
Due to Corollary \ref{thrm,ehl}, their images $f(R_j)$ are extreme half-lines in $Y$ and they are distinct as $f$ is injective. For each $\lambda\ge 0$, the set $x+S+\lambda r$ is an extreme half-line that intersects $R_j$ for each $j\in \{0,1,2\}$, so, by Corollary \ref{thrm,ehl}, $f(x+S+\lambda r)$ is an extreme half-line  and 
\begin{equation}\label{eq.itintersects}
f(x+S+\lambda r) \mbox{ intersects }f(R_j) \mbox{ for each }j\in\{0,1,2\}\mbox{ and }\lambda\ge 0.
\end{equation} 
We obtain that $f(x+S+\lambda r)$ is not parallel to any of the $f(R_j)$, as $R$ and $S$ are distinct and $f$ is injective.

We aim to show that $f(R_0)$, $f(R_1)$, and $f(R_2)$ are parallel. We do so in two steps. As a first step we show that if two of them are parallel, then all three of them are parallel. Indeed, assume that $f(R_j)$ and $f(R_k)$ are parallel, with $j,k\in\{0,1,2\}$, $j\neq k$. Since $f(R_j)$ and $f(R_k)$ are distinct parallel half-lines, it follows from \eqref{eq.itintersects} that the half-line $f(x+S+\lambda r)$ is in their affine span for every $\lambda\ge 0$. Then the half-line $f(R_i)$ with $i\in \{0,1,2\}\setminus \{j,k\}$ is in that affine span, too, as it intersects $f(x+S+\lambda r)$ for two distinct values of $\lambda$. Thus, $f(x+S)$, $f(R_i)$, and $f(R_j)$  are three extreme half-lines in the affine plane spanned by $f(R_j)$ and $f(R_k)$. By Lemma \ref{lem,cad}, it follows that at least two of the half-lines $f(x+S)$, $f(R_i)$, and $f(R_j)$ must be parallel, which yields that $f(R_i)$ and $f(R_j)$ must be parallel. Thus, $f(R_i)$, $f(R_j)$, and $f(R_k)$ are parallel.

As a second step we argue by contradiction that at least two of the half-lines $f(R_0)$, $f(R_1)$, and $f(R_2)$ are parallel. For $i\in\{0,1,2\}$, take $w_i\in Y$ such that 
\[f(R_i)=\{f(x+is)+\lambda w_i\colon\, \lambda\ge 0\}.\]
Suppose that no two of the three extreme half-lines $f(R_0)$, $f(R_1)$, and $f(R_2)$ are parallel. After translation they correspond to three distinct extremal rays, so that Lemma \ref{lem,cad} yields that $w_0$, $w_1$, and $w_2$ are linearly independent.
Define 
\begin{align*}
W_0&=f(x)+\operatorname{span}\{w_0,w_2\},\\  
W_2&=f(x+2s)+\operatorname{span}\{w_0,w_2\},\\ 
\ell_1&=\{f(x+s)+\lambda w_1\colon\, \lambda\in\mathbb{R}\}.
\end{align*}

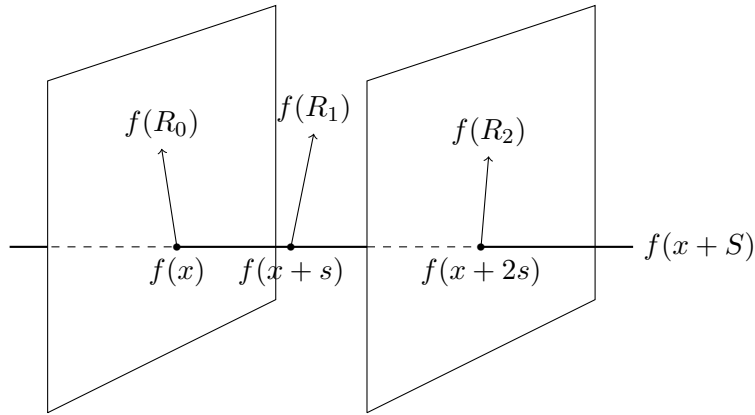
\begin{center}
\begin{tikzpicture}
\draw [] (2.3,-2.2) -- (5.3,-0.7) -- (5.3,3.2) -- (2.3,2.2) -- (2.3,-2.2);
\draw [] (6.5,-2.2) -- (9.5,-0.7) -- (9.5,3.2) -- (6.5,2.2) -- (6.5,-2.2);
\draw [thick] (1.8,0) -- (2.3,0);
\draw [dashed] (2.35,0) -- (4,0);
\draw [thick] (4,0) -- (6.5,0);
\draw [dashed] (6.5,0) -- (8,0);
\draw [thick] (8,0) -- (10,0);
\draw node[right] at (10,0) {$f(x+S)$};
\draw node at (4,0) {\tiny{$\bullet$}};
\draw node[below] at (4,0) {$f(x)$};
\draw node at (5.5,0) {\tiny{$\bullet$}};
\draw node[below] at (5.5,0) {$f(x+s)$};
\draw node at (8,0) {\tiny{$\bullet$}};
\draw node[below] at (8,0) {$f(x+2s)$};
\draw [->] (4,0) -- (3.8,1.3);
\draw node[above] at (3.8,1.3) {$f(R_0)$};
\draw [->] (5.5,0) -- (5.8,1.5);
\draw node[above] at (5.8,1.5) {$f(R_1)$};
\draw [->] (8,0) -- (8.1,1.2);
\draw node[above] at (8.1,1.2) {$f(R_2)$};
\end{tikzpicture}
\end{center}
We observe that $W_0$ and $W_2$ are parallel and distinct planes. Moreover, $f(R_0)\subseteq W_0$, $f(R_2)\subseteq W_2$ and $f(R_1)\subseteq \ell_1$. The affine span $\operatorname{aff}(W_0,W_2)$ of $W_0$ and $W_2$ is three dimensional and contains $\ell_1$. Indeed, for every $z\in f(R_1)$ there is $\lambda\ge 0$ with $z=f(x+s+\lambda r)$, and by \eqref{eq.itintersects}, $\operatorname{aff}(W_0,W_2)$ contains the half-line $f(x+S+\lambda r)$. This shows that $f(R_1)\subseteq \operatorname{aff}(W_0,W_2)$, and hence $\ell_1\subseteq \operatorname{aff}(W_0,W_2)$. Since $w_1$ is linearly independent of $w_0$ and $w_2$, we conclude that $\ell_1$ intersects $W_0$ and $W_2$.

We proceed by showing that the half-line $f(R_1)$ intersects $W_0$ or $W_2$. Loosely speaking, the point $f(x+s)$ on $\ell_1$ lies between $W_0$ and $W_2$ and, therefore, the points where $\ell_1$ intersects $W_0$ and $W_2$ cannot be both at the same side of $f(x+s)$. To make this idea precise, let $v\in Y$ be such that
\[f(x+S)=\{f(x)+\lambda v\colon\, \lambda\ge 0\}.\]
Observe that $v\in K$, as $f(x+S)\subseteq [f(x),\infty)$. Then
\[\operatorname{aff}(W_0,W_2)=\{f(x+s)+\lambda w_0+\mu w_2+\sigma v\colon\, \lambda,\mu,\sigma\in \mathbb{R}\}.\]
As $f(x+s)+w_1\in f(R_1)\subseteq \operatorname{aff}(W_0,W_2)$, there are $\lambda,\mu,\sigma\in\mathbb{R}$ such that $w_1=\lambda w_0+\mu w_2+\sigma v$. By linear independence of $w_0$, $w_1$ and $w_2$, we have $\sigma\neq 0$. Consider the case $\sigma<0$. Then $f(R_1)$ intersects $W_0$, 
so there is a $t> 0$ such that $f(x+s+tr)\in W_0$. As $f(x+R)=f(R_0)\subseteq W_0$, it follows that the half-line $f(x+S+tr)$ contains two distinct points of $W_0$, so that $f(x+S+tr)\subseteq W_0$. Therefore $f(x+2s+tr)\in W_0\cap f(R_2)\subseteq W_0\cap W_2$, which is a contradiction. Otherwise, in case $\sigma>0$, then $f(R_1)$ intersects $W_2$, and we similarly arrive at a contradiction. Hence at least two of the half-lines $f(R_0)$, $f(R_1)$, and $f(R_2)$ are parallel, so by the first step all three of them are parallel.

Now we complete the proof. As $f(R_0)$ and $f(R_1)$ are parallel, we have that the vectors $f(x+r)-f(x)$ and $f(x+s+r)-f(x+s)$ have the same direction. By interchanging the roles of $R$ and $S$ we obtain that the vectors $f(x+s)-f(x)$ and $f(x+s+r)-f(x+r)$ have the same direction. Thus, $f(x)$, $f(x+r)$, $f(x+s+r)$, and $f(x+s)$ are the consecutive corners of a parallellogram, which concludes the proof.
\end{proof}

 It is interesting to note that the  proof of Lemma \ref{lem,add} does not work if the domain of the order-isomorphism is bounded. In fact, there exist examples of  order-isomorphisms on bounded order intervals for which  equation (\ref{eq:par}) does not hold, see for example \cite{Sem} where order-isomorphisms on order intervals in $B(H)_{\mathrm{sa}}$ are studied.

The following observation is a simple consequence of the previous  lemma.
\begin{corollary}\label{c:negative extreme vectors} Suppose $r,s\in X$ are extreme vectors with $r\neq \lambda s$ for all $\lambda\in\mathbb{R}$ and $f\colon [a,\infty) \to [b,\infty)$ is an order-isomorphism. If $x\in [a,\infty)$ is such that $x+r+s,x+r,x+s\in[a,\infty)$ then \[
f(x+r+s)-f(x+r)=f(x+s)-f(x).\]\end{corollary}
\begin{proof} 
We only discuss the proof for the case $r\leq 0$ and $s\leq 0$, and leave the other two remaining cases to the reader, as they are proved in a similar way. By writing $y=x+r+s$, we get 
\[ f(x+r+s)-f(x+s) = f(y)-f(y-r) = f(y-s)-f(y-r-s) = f(x+r)-f(x) \] 
by Lemma \ref{lem,add}.
\end{proof}

Using this corollary we now show the following lemma.
\begin{lemma}\label{l:finite sum out} Let $f\colon [a,\infty)\to[b,\infty)$ be an order-isomorphism. Suppose $s_1,\ldots,s_n,r\in X$ are extreme vectors  such that $r\neq\lambda s_i$ for all $\lambda\in\mathbb{R}$ and $i=1,\ldots,n$. If $x,x+s_1+\cdots +s_n+r,x+s_1+\cdots +s_n,x+r\in [a,\infty)$, then \[
f\left(x+r+\sum_{i=1}^n s_i\right)-f\left(x+\sum_{i=1}^n s_i\right)=f(x+r)-f(x).\]\end{lemma}
\begin{proof}
By relabelling we may assume that there exists $k\in\{0,\ldots,n\}$ such that $s_i>0$ for all $i\leq k$ and $s_i<0$ for all $i>k$. Then $x+r+\sum_{i=1}^m s_i\in [a,\infty)$ and $x+\sum_{i=1}^m s_i \in[a,\infty)$ for $m=1,\ldots,n$. By Corollary \ref{c:negative extreme vectors} we have 
\[f\left((x+\sum_{i=1}^{n-1} s_i)+s_n+r\right)-f\left((x+\sum_{i=1}^{n-1} s_i)+s_n\right)=f\left(x+\sum_{i=1}^{n-1} s_i+r\right)-f\left(x+\sum_{i=1}^{n-1} s_i\right).\]
Repeating this argument yields the desired conclusion.
\end{proof}
We can use Lemma \ref{l:finite sum out} to get the following identity.
\begin{lemma}\label{kerstboom} Let $f\colon [a,\infty)\to[b,\infty)$ be an order-isomorphism. Suppose $x\in [a,\infty)$ and $s_1,\ldots,s_n$ are extreme vectors in $X$ such that $s_i\neq \lambda s_j$ for all $\lambda \in\mathbb{R}$ and $i\neq j$, $x+s_1+\cdots+s_n\in [a,\infty)$, and $x+s_i\in [a,\infty)$ for all $i=1,\ldots,n$, then 
\[
f\left(x+\sum_{i=1}^n s_i\right)-f(x)=\sum_{i=1}^n \left(f(x+s_i)-f(x)\right).
\]
\end{lemma}
\begin{proof}
By relabelling we may assume that there exists $k\in\{0,\ldots,n\}$ such that $s_i>0$ for all $i\leq k$ and $s_i<0$ for all $i>k$. Then $x+\sum_{i=1}^m s_i \in[a,\infty)$ for $m=1,\ldots,n$.  
Using  a telesoping sum and Lemma \ref{l:finite sum out} we obtain 
\[ f\left(x+\sum_{i=1}^n s_i\right)-f(x)  =  f\left( x +\sum_{i=1}^n s_i\right) -  f\left( x +\sum_{i=1}^{n-1}  s_i \right)   + \cdots +  f(x+s_1) - f(x) 
 =  \sum_{m=1}^n (f(x+s_m)-f(x)).
  \]
\end{proof}

Let $\mathcal{R}$ denote the collection of all extreme rays in $C$, and define
\[
[a,\infty)_\mathcal{R} =\{a+r_1+\cdots+r_n\in [a,\infty) \colon r_i\in C \text{ an extreme vector for } i=1,\ldots,n\}.\]

\begin{lemma}\label{l:invariant r} Let $f\colon [a,\infty)\to[b,\infty)$ be an order-isomorphism and $x,y\in[a,\infty)_{\mathcal{R}}$. Suppose that  $y-x=s_1+\cdots+s_n$, where $s_i\in X$ is an extreme vector for $i=1,\ldots,n$. If $r\in X$ is an extreme vector with $r\neq\lambda s_i$ for all $\lambda\in\mathbb{R}$ and $i=1,\ldots,n$, and $x+r,y+r\in [a,\infty)$, then 
\[f(x+r)-f(x)=f(y+r)-f(y).\]
\end{lemma}
\begin{proof} Note that  \[f(y+r)-f(y) = f(x+(y-x)+r)-f(x+(y-x)) = f(x+s_1+\cdots+s_n+r)-f(x+s_1+\cdots+s_n) = f(x+r)-f(x)\] by Lemma \ref{l:finite sum out}.\end{proof}

In the setting of Lemma \ref{l:invariant r}, if $r=\lambda s_i$ for some $\lambda$ and $i$, and $r\in\mathrm{span} \{s\colon\, s\in S\text{ and } S\in\mathcal{R}\setminus \{R\}\}$ where $R=\{\lambda r\colon\, \lambda\ge 0\}$, then one could replace $s_i$ by a linear combination of extreme vectors not contained in $R\cup-R$ and thus obtain $y-x=s_1'+\cdots+s_m'$ with $r\neq \lambda s_j'$ for all $\lambda$ and $j$. Then the conclusion of Lemma \ref{l:invariant r} still holds. This motivates the following definition due to \cite{NS2} . 
\begin{definition}\label{defn,eng} Let $\mathcal{S}$ be a collection of rays in a cone $C$ in a vector space $X$. A ray $R\in\mathcal{S}$ is called {\it engaged (in $\mathcal{S}$)} whenever 
\[R\subseteq  \mathrm{span}(\mathcal{S}\setminus\{R\})=\mathrm{span}\{s\colon\, s\in S\text{ and } S\in\mathcal{S}\backslash \{R\}\}\] 
holds, and $R$ is called {\it disengaged (in  $\mathcal{S}$)} otherwise.\end{definition}
It can be shown that an extreme ray of a finite dimensional cone is disengaged (in the set of extreme rays) if and only if the cone equals the Cartesian product of the ray and another subcone. Cones that do not allow such a decomposition are considered in \cite{A}.

Recall that $\mathcal{R}$ denotes the collection of all extreme rays of $C$.
We denote the collection of all engaged extreme rays in $\mathcal{R}$ by $\mathcal{R}_E$ and the collection of all disengaged extreme rays in $\mathcal{R}$  by $\mathcal{R}_D$.
We remark that  being an engaged ray is relative to the collection it is viewed in. Nevertheless, we have that the elements of $\mathcal{R}_E$ are again engaged in $\mathcal{R}_E$. 
For simplicity we say that an extreme vector $r\in R\cup -R$ is {\em engaged} if $R\in\mathcal{R}_E$.

\begin{lemma}\label{lem,enginv} If $r\in X$ is an extreme vector, then the following assertions hold:
\begin{enumerate}[(i)]
\item $f(x+\lambda r)-f(x)$ is a scalar multiple of $f(x+r)-f(x)$ for every $x\in [a,\infty)$ and $\lambda\in\mathbb{R}$ such that $x+r,x+\lambda r\in [a,\infty)$;
\item If $r$ is engaged and $x,y,x+r,y+r\in [a,\infty)$ and $y-x\in\mathrm{span}\,\mathcal{R}$, then \[f(x+r)-f(x)=f(y+r)-f(y).\]
\end{enumerate}\end{lemma}
\begin{proof}
Assertion (i) follows from Corollary \ref{thrm,ehl}.
Remark that if $r$ is engaged then there exist extreme vectors $s_1,\ldots,s_n$ with $y-x=s_1+\cdots+s_n$ such that $r\neq\lambda s_i$ for all $\lambda\in\mathbb{R}$ and $i=1,\ldots,n$. So (ii) follows from Lemma \ref{l:invariant r}.
\end{proof}

The following result is an extension of \cite[Theorem A]{NS2}.  Recall that $\mathcal{R}_E$ denotes the collection of engaged extreme rays in $\mathcal{R}$. We define \[
[a,\infty)_{\mathcal{R}_E}=\{a+r_1+\cdots+r_n\in[a,\infty)\colon r_i\in C \text{ an engaged extreme vector for } i=1,\ldots,n\}.\]

\begin{theorem}\label{thrm,ns}  If $f\colon [a,\infty)\to [b,\infty)$ is an order-isomorphism, then $f$ is affine on $[a,\infty)_{\mathcal{R}_E}$.\end{theorem}
\begin{proof}
Let $R$ be an engaged extreme ray of $C$ and fix $r\in R\backslash\{0\}$. Let $\lambda\in\mathbb{R}$ and take $x\in[a,\infty)_{\mathcal{R}}$ such that $x+\lambda r\geq a$. Then $x,x+r,x+\lambda r\in [a,\infty)$. So, by Lemma \ref{lem,enginv}(i), there exists a unique $g_{r,x}(\lambda)\in\mathbb{R}$ such that 
\begin{equation}\label{gr0}
f(x+\lambda r)-f(x)=g_{r,x}(\lambda)(f(x+r)-f(x)).
\end{equation}
As $r$ is engaged, it follows from Lemma \ref{lem,enginv}(ii) that $g_{r,x}(\lambda)$ does not depend on $x$. Thus there exists a unique function $g_r\colon\mathbb{R}\to\mathbb{R}$ such that for every $\lambda\in\mathbb{R}$ and $x\in[a,\infty)_{\mathcal{R}}$ with $x+\lambda r\geq a$ we have 
\begin{equation}\label{gr}
f(x+\lambda r)-f(x)=g_r(\lambda)(f(x+r)-f(x)).
\end{equation}
Clearly, $g_r(1)=1$ and $g_r$ is a monotone increasing function. For $\lambda,\mu\in\mathbb{R}$ there exists  an $x\in[a,\infty)_{\mathcal{R}}$ such that  $x+\lambda r\geq a$, $x+\mu r\geq a$, and $x+\lambda r+\mu r\geq a$. Moreover 
\begin{align*} g_r(\lambda+\mu)(f(x+r)-f(x)) &= f(x+(\lambda+\mu)r)-f(x) \\&= f(x+\lambda r+\mu r)-f(x+\lambda r)+f(x+\lambda r)-f(x) \\&= g_r(\mu)(f(x+\lambda r +r)-f(x+\lambda r))+g_r(\lambda)(f(x+r)-f(x)).\end{align*}
Since $r$ is engaged, Lemma \ref{lem,enginv}(ii) gives
\[
f(x+\lambda r+r)-f(x+\lambda r)=f(x+r)-f(x).
\]
Note that $f(x+r)-f(x)\neq 0$, as $r\neq 0$ and $f$ is injective, and hence  
\[
g_r(\lambda+\mu)=g_r(\lambda)+g_r(\mu).
\]
As $g_r$ is monotone increasing, additive, and $g_r(1)=1$, a result by Darboux (see \cite[Theorem 1 in Section 2.1]{Aczel}) yields that $g_r(\lambda)=\lambda$ for all $\lambda\in\mathbb{R}$.

To show that $f$ is affine it suffices to show that $f$ is convex-linear on $[a,\infty)_{\mathcal{R}_E}$, let $x,y\in [a,\infty)_{\mathcal{R}_E}$ and $0\leq t\leq 1$. Then $x=a+\sum_{i=1}^n \lambda_i r_i$ and $y=a+\sum_{i=1}^n \mu_i r_i$ where each $r_i\in C$ is an engaged extreme vector  and $r_i\neq \lambda r_j$ for all $\lambda\in\mathbb{R}$ and $i\neq j$. Moreover, $\lambda_i,\mu_i\geq 0$ and $\lambda_i+\mu_i\neq 0$ for all $i$. Put $s_i=(t\lambda_i+(1-t)\mu_i)r_i$. As $a+s_i\in [a,\infty)$ for all $i$, we can apply  Lemma \ref{kerstboom} to get 
\begin{align} 
f(tx+(1-t)y)-f(a) &= f\left(a+\sum_{i=1}^n s_i\right)-f(a)
                          = \sum_{i=1}^n ( f(a+s_i)-f(a))\nonumber\\
                          &= \sum_{i=1}^n (f(a+(t\lambda_i +(1-t)\mu_i)r_i)-f(a))
                          = \sum_{i=1}^n (t\lambda_i+(1-t)\mu_i)(f(a+r_i)-f(a))\nonumber\\
                          &= t\sum_{i=1}^n \lambda_i(f(a+r_i)-f(a))+(1-t)\sum_{i=1}^n\mu_i(f(a+r_i)-f(a))\nonumber\\
                          &= t\sum_{i=1}^n (f(a+\lambda_ir_i)-f(a))+(1-t)\sum_{i=1}^n (f(a+\mu_ir_i)-f(a))\nonumber\\
                          &= t(f\left(a+\sum_{i=1}^n \lambda_i r_i\right)-f(a))+(1-t)(f\left(a+\sum_{i=1}^n \mu_ir_i\right)-f(a)) \nonumber
                         \\&= tf(x)+(1-t)f(y)-f(a),\label{eq,star}\end{align}
where we have used \eqref{gr} and the fact that each $r_i$ is engaged in the forth and sixth equality, and Lemma \ref{kerstboom} in the seventh one. This completes the proof.
\end{proof}

\begin{remark}
It is interesting to note that in the proof of Theorem \ref{thrm,ns} we have only  used the assumption that  $r$ is an engaged extreme vector to show that the map $g_r\colon\mathbb{R}\to\mathbb{R}$ satisfying \eqref{gr0} is independent of $x$ and additive.  However,  if $r$ is a disengaged extreme vector, then  \eqref{gr0} still holds.  In Section 5 we will exploit this observation. Moreover, we remark that it is necessary to work with  the positive linear span of engaged extreme vectors, $[a,\infty)_{\mathcal{R}_E}$. Indeed, to apply Lemma \ref{kerstboom} we need   for each $i$  that $a+s_i$ is in the domain of $f$. 
\end{remark}

Let us now see how we can use Theorem \ref{thrm,ns} to generalise \cite[Theorem A]{NS2}. 
Suppose $V\subseteq X$ and $\sup V$ exists. If $f$ is an order-isomorphisms, then $f(\sup V)=\sup f(V)$. Likewise  order-isomorphisms preserve infima.  
These basic observations motivate the following definition. 
\begin{definition}\label{defn,supinf} Given $V\subseteq X$ the {\it inf-sup hull of} $V$ is the set $$\{x\in X\colon \mbox{ there exist } v_{\alpha,\beta}\in V\mbox{ for }\alpha \in A \mbox{ and } \beta \in B \mbox{ such that } x=\inf_{\alpha\in A} (\sup_{\beta\in B} v_{\alpha,\beta})\}, $$
where $A$ and $B$ are arbitrary index sets.
\end{definition}
Note that if $V\subseteq X$  and $x$ and $y$ are in the inf-sup hull of $V$, then $x=\inf_{\alpha\in A}(\sup_{\beta\in B} x_{\alpha,\beta})$ and $y=\inf_{\sigma \in S }(\sup_{\tau \in T} y_{\sigma,\tau})$, with all $x_{\alpha,\beta}$ and $y_{\sigma,\tau}$ in $V$, and hence for all  $\lambda,\mu\geq 0$ we have that 
\begin{align} \lambda x+\mu y &=\inf_{\alpha\in A}(\sup_{\beta\in B} \lambda x_{\alpha,\beta}) + \inf_{\sigma\in S}(\sup_{\tau\in T} \mu y_{\sigma,\tau})\nonumber
= \inf_{\alpha\in A}(\sup_{\beta\in B} \lambda x_{\alpha,\beta} + \inf_{\sigma\in S}(\sup_{\tau\in T} \mu y_{\sigma,\tau}))\nonumber\\
&=\inf_{\alpha\in A}(\inf_{\sigma\in S}(\sup_{\beta\in B}  \lambda x_{\alpha,\beta} +\sup_{\tau\in T}\mu y_{\sigma,\tau}))\nonumber
=\inf_{\alpha\in A}(\inf_{\sigma\in S}(\sup_{\beta\in B}(\sup_{\tau\in T} \lambda x_{\alpha,\beta} +\mu y_{\sigma,\tau})))\nonumber\\
&= \inf_{(\alpha,\sigma)\in A\times S} (\sup_{(\beta,\tau)\in B\times T} \lambda x_{\alpha,\beta} + \mu y_{\sigma,\tau}),\label{eq,infsup}
\end{align} 
which shows that $\lambda x+\mu y$ is also in the inf-sup hull. In particular we see that the inf-sup hull of a convex subset of $X$ is again a convex set.

\begin{lemma}\label{lem,extend}  
Let $f\colon [a,\infty) \to [b,\infty)$ be an order-isomorphism and let $D\subseteq [a,\infty)$ be convex. If $f$ is affine on $D$, then $f$ is affine on the inf-sup hull of $D$.
\end{lemma}
\begin{proof} Suppose $V\subseteq [a,\infty)$ and $v\in [a,\infty)$ are such that $v=\sup(V)$. Then $f(v)$ is an upper bound of $f(V)$ in $[b,\infty)$.
Moreover, if $w\in [b,\infty)$ is another upper bound of $f(V)$, then $f^{-1}(w)$ is an upper bound of $V$, since $f^{-1}$ is order preserving.
As  $v=\sup(V)$ we deduce that $v\leq f^{-1}(w)$, so  that $f(v)\leq w$. This implies that  $f(v)=\sup(f(V))$ in $[b,\infty)$.
In the same way it can be shown that if  $W\subseteq [a,\infty)$ and $w\in [a,\infty)$ are such that $w=\inf(W)$, then $f(w)=\inf(f(W))$ in $[b,\infty)$.

To complete the proof it  suffices to show that $f$ is convex-linear on  the inf-sup hull of $D$. Indeed, the inf-sup hull of $D$ is a convex set by \eqref{eq,infsup}. Suppose that $x$ and $y$ are in the inf-sup hull of $D$ and $0\leq t\leq 1$. Write $x=\inf_\alpha\sup_\beta x_{\alpha,\beta}$ and $y =\inf_\sigma\sup_\tau y_{\sigma,\tau}$, with $x_{\alpha,\beta},y_{\sigma,\tau} \in D$ for all $\alpha,\beta,\sigma$ and $\tau$.

By repeatedly using the fact that $f$ preserves infima and suprema and Theorem \ref{thrm,ns} we get  
\begin{align*} f(tx+(1-t)y) &= \inf_{\alpha\in A}(\sup_{\beta\in B}(\inf_{\sigma\in S}(\sup_{\tau\in T} f(tx_{\alpha,\beta}+(1-t)y_{\sigma,\tau})))) \\& = \inf_{\alpha\in A}(\sup_{\beta\in B}(\inf_{\sigma\in S}(\sup_{\tau\in T} tf(x_{\alpha,\beta})+(1-t)f(y_{\sigma,\tau})))) \\ &= tf(\inf_{\alpha\in A}(\sup_{\beta\in B} x_{\alpha,\beta})) + (1-t)f(\inf_{\sigma\in S}(\sup_{\tau\in T} y_{\sigma,\tau})) = tf(x)+(1-t)f(y).\end{align*}
 \end{proof}

Combination of Theorem \ref{thrm,ns} and Lemma \ref{lem,extend} yields the next conclusion.

\begin{proposition}\label{prop,ns}  
Every  order-isomorphism $f\colon [a,\infty) \to [b,\infty)$ is affine on the inf-sup hull of $[a,\infty)_{\mathcal{R}_E}$.
\end{proposition}

We can now prove our main result Theorem \ref{thrm,ext}.
\begin{proof}[Proof of Theorem \ref{thrm,ext}] Let $a\in U$ be given. As $C$ is the inf-sup hull of $ [0,\infty)_{\mathcal{R}_E}$, we get that the interval $[a,\infty)$ equals the inf-sup hull of $[a,\infty)_{\mathcal{R}_E}$.
So it follows from Proposition \ref{prop,ns} that $f$ is affine on $[a,\infty)$.
As $X$ is directed the cone $C$ is generating, and hence $C-C = X$. This implies that there exists a unique  affine map $g:X\rightarrow Y$ such that $g$ restricted to $[a,\infty)$ coincides with $f$. 

In the same way we find that for  any $b\in U$ the map $f$ is affine on $[b,\infty)$.
Using that $C$ is directed, we know there exists $c\in U$ such that $c\geq a,b$.
We remark that the intersection $[a,\infty)\cap [b,\infty)$ contains the interval $[c,\infty)$.
Therefore, $f$ and $g$  coincide on $[b,\infty)$ for all $b\in U$. Since $U =\bigcup_{b\in U} [b,\infty)$,  we conclude that $g$ coincides with $f$ on $U$, which completes the proof. \end{proof}

Theorem \ref{thrm,ext} is a generalisation of  \cite[Theorem A]{NS2} by Noll and Sch\"{a}ffer. It would be  interesting to  have a complete characterisation of the  (infinite dimensional) directed Archimedean partially ordered vector spaces $(X,C)$ for which every  order-isomorphism $f\colon C\to C$ is linear. To our knowledge, Theorem \ref{thrm,ext} is the most general result at present. It can, however, not be applied in a variety of settings  such as the space $C([0,1])\oplus \mathbb{R}$ with cone $\{(f,\alpha)\colon \|f\|_\infty\leq \alpha\}$. In this space the cone has exactly two disengaged extreme rays: $\{\lambda(\mathds{1},1)\colon \lambda\geq 0\}$ and  $\{\lambda(-\mathds{1},1)\colon \lambda\geq 0\}$, where $\mathds{1}(x)=1$ for all $x\in [0,1]$, but it has no engaged extreme rays. We believe, however, that  each order-isomorphism on the cone is linear in this space. 

We end this section with a simple observation concerning direct sums. Let $(X_1,C_1)$ and $(X_2,C_2)$ be directed  Archimedean partially ordered vector spaces. Then the direct sum $X_1\oplus X_2$ is a directed  Archimedean partially ordered vector space with  cone $C_1\times C_2$. Moreover $(r,s)\in C_1\times C_2$ is an (engaged) extreme vector if and only if $r$ is an (engaged) extreme vector and $s=0$, or, $s$ is an (engaged) extreme vector and $r=0$. It is straightforward to infer that if $(X_1,C_1)$ and $(X_2,C_2)$ satisfy the conditions on $(X,C)$ in  Theorem \ref{thrm,ext}, then so does $(X_1\oplus X_2,C_1\times C_2)$.


\section{Self-adjoint operators on a Hilbert space}
Let $H$ be a Hilbert space and $B(H)_{\mathrm{sa}}$ be the space of bounded self-adjoint operators on $H$, ordered by the cone $B(H)_{\mathrm{sa}}^+$ of positive semi-definite operators.
In this section we show that $B(H)_{\mathrm{sa}}$ satisfies the conditions of Theorem \ref{thrm,ext}.

It is easy to show that the extreme rays of $B(H)_{\mathrm{sa}}^+$ are the rays spanned by rank-one  projections.
We will denote the collection of all extreme rays of $B(H)_{\mathrm{sa}}^+$ by $\mathcal{R}$.
Furthermore, for a closed subspace $V$ of $H$ we denote the orthogonal projection onto by $V$ by $P_V$, and for $x\in H$ we write $P_x=P_{\text{span}(\{x\})}$.

\begin{theorem}\label{thrm,bh}  If  $H$ is a Hilbert space, with $\dim H\geq 2$, and $U,W\subseteq B(H)_{\mathrm{sa}}$ are upper sets, then every  order-isomorphism $f\colon U\to W$  is affine.\end{theorem}
\begin{proof} 
We verify that $B(H)_{\mathrm{sa}}$ satisfies the conditions of Theorem \ref{thrm,ext}. Evidently, $B(H)_{\mathrm{sa}}$ is directed and Archimedean.
We first show that all extreme rays of $B(H)_{\mathrm{sa}}^+$ are engaged. So, suppose $P\in \mathcal{R}$. Then there exists an $x\in H$ such that $P=P_x$.
As $\dim H\geq 2$ we can find non-zero $y,z\in H$ such that $y$ and $z$ are orthogonal and $x,y,z$ lie in a two-dimensional subspace $V$.
Then  $P_V=P_y+P_z$, so that  
\[
P_x=P_V-(I-P_x)P_V=P_y+P_z-P_{\{x\}^\perp} P_V=P_y+P_z-P_w,
\] 
where $w\in\{x\}^\perp \cap (V\backslash\{0\})$.
We conclude that $P_x$ can be written as a linear combination of rank-one projections different from $P_x$ and, hence, the ray spanned by $P_x$ is engaged in $\mathcal{R}$.

Note that the positive linear span of the extreme rays equals the set of positive finite rank operators, which will be denoted $F$.
 To verify the condition in Theorem \ref{thrm,ext} it suffices to show that the inf-sup hull of $F$ equals $B(H)^+_{\mathrm{sa}}$, as the inf-sup hull is closed under  positive sums by (\ref{eq,infsup}).

We start by showing that  the identity $I$ belongs to the inf-sup hull of $F$.
Note that $I\geq P_x$ for all $x\in H$.
Suppose that $B\in B(H)_{\mathrm{sa}}$ is an upper bound of $P_x$ for all $x\in H$.
Then we have for any $x\in H$ that
 \begin{equation}\label{eq,id}\langle Bx,x\rangle \geq \langle P_x x,x\rangle =\langle Ix,x\rangle.\end{equation}
Therefore, $B\geq I$ holds and we conclude that $I=\sup\{P_x\colon x\in H\}.$
Note that it follows from (\ref{eq,id}) that for each $Q_0\in F$ with $Q_0\leq I$ we have that $$I=\sup\{Q\in F\colon Q_0\leq Q\leq I\},$$ as for all $x\in H$ there exists a $Q\in F$ with $Q_0\leq Q\leq I$ and $Q\geq P_x$.

Now suppose that $A\in B(H)_{\mathrm{sa}}^+$ is invertible. Let $T_A\colon B(H)_{\mathrm{sa}}\to B(H)_{\mathrm{sa}}$ be given  by $T_A(Q)=A^{\frac{1}{2}}QA^{\frac{1}{2}}$.
Then  $T_A$ is a linear order-isomorphism, so that 
\[A=T_A(I) = T_A(\sup\{Q\in F\colon  Q_0\leq Q\leq I\})= \sup\{T_A(Q): Q\in F,\;Q_0\leq Q\leq I\}.\]
As $T_A$ is a bijection from $F$ onto itself, we get that $A=\sup\{Q\in F: T_A(Q_0)\leq Q\leq A\}$.

Finally, suppose $A\in B(H)^+_{\mathrm{sa}}$. Remark that $A+I$ is invertible. For  $P\in F$, with $P\leq I$ we let $Q_0=T_{(A+I)^{-1}}(P)$.
Then $A+I=\sup\{Q\in F\colon  T_{A+I}(Q_0)\leq Q\leq A+I\}$, from which it follows that  $A+I-P=\sup\{Q-P\colon Q\in F,\; P\leq Q\leq A+I\}$. Thus, 
\[A = \inf \{A+I-P\colon P\in F, P\leq I\} = \inf\{\sup\{Q-P\colon  Q\in F,\;P\leq Q\leq A+I\}\colon P\in F,\;P\leq I\}.\]

This shows that $B(H)_{\mathrm{sa}}^+$ is the inf-sup hull of the positive linear span of its  extreme rays, and hence Theorem \ref{thrm,ext} yields the desired result.
\end{proof} 

We remark that Theorem \ref{thrm,bh} was first proved, using different arguments, by Moln\'ar  \cite{M} and does not follow from \cite[Theorem A]{NS2}.  

\section{Order-isomorphisms in related problems}

In this section we proceed the discussion of Section \ref{section3} and relate to results by Artstein-Avidan and Slomka and Sch\"affer in settings somewhat different than in Theorem \ref{thrm,ext}. We obtain three results. First, we present a ``diagonalization formula'' for order-isomorphisms between cones, see \eqref{eq,diag} below. Second, we apply the results of Section \ref{section3} to positively homogeneous order-isomorphisms between cones and obtain that they must be linear if one of the cones equals the inf-sup hull of the positive span of its extreme rays. Third, we consider separable complete order unit spaces where in one of them the inf-sup hull of the positive linear span of the engaged extreme rays is big enough to intersect the interior of the cone. In that case we derive from Theorem \ref{thrm,ext} that every order-isomorphism between upper sets must be affine.

We begin with the following infinite dimensional analogue of a result by Artstein-Avidan and Slomka \cite[Theorem~1.7]{AAS}. 

\begin{proposition}\label{prop,diag} 
Let $(X,C)$ and $(Y,K)$ be Archimedean partially ordered vector spaces and suppose that $f\colon C\to K$ is an order-isomorphism. Let $(v_\alpha)_{\alpha\in A}$ be a collection of linearly independent extreme vectors in $C$. Then there exist corresponding monotone increasing bijections $g_\alpha\colon [0,\infty)\to [0,\infty)$, for $\alpha\in A$, such that for all $\lambda_1,\ldots,\lambda_n\geq 0$ and $\alpha_1,\ldots,\alpha_n\in A$ we have
\begin{equation}\label{eq,diag} 
f\left(\sum_{i=1}^n \lambda_i v_{\alpha_i}\right)=\sum_{i=1}^n g_{\alpha_i}(\lambda_i)f(v_{\alpha_i}).
\end{equation}
\end{proposition}
\begin{proof} 
Note that $f(0)=0$. Let  $r\in C$ be an extreme vector. According to Corollary \ref{thrm,ehl}, $f$ maps the extreme ray through $r$ bijectively onto the extreme ray through $f(r)$. Hence there exists a nonnegative scalar $g_{r}(\lambda)$ such that $f(\lambda r)=g_r(\lambda)f(r)$, for all $\lambda\geq 0$. Moreover, the function $g_{r}\colon [0,\infty)\to [0,\infty)$ is a monotone increasing bijection. Equation \eqref{eq,diag} now follows from Lemma \ref{kerstboom}.
\end{proof}
In \cite[Theorem~1.7]{AAS}, also the finite dimensional cases $f\colon X\to X$ and $f\colon C^\circ\to C^\circ$ are considered. In the situation of Proposition \ref{prop,diag}, if $f$ is an order-isomorphism from $X$ to $Y$ and $f(0)=0$, then one can easily verify that the maps $g_{r}$ are actually defined on $\mathbb{R}$ and that \eqref{eq,diag} holds for all $\lambda\in\mathbb{R}$. The infinite dimensional version of the case where  $f\colon C^\circ\to K^\circ$ is not so strong. Indeed, if $(X,C)$ and $(Y,K)$ are infinite dimensional order unit spaces, then one can adapt the proof of Proposition \ref{prop,diag} to show that for each order-isomorphism $f\colon C^\circ\to K^\circ$ and each collection $(v_\alpha)_{\alpha\in A}$ of linearly independent extreme vectors of $C$, there are linearly independent extreme vectors $(w_\alpha)_{\alpha\in A}$ of $K$ and monotone increasing bijections $g_\alpha\colon [0,\infty)\to [0,\infty)$, $\alpha\in A$, such that for all $\lambda_1,\ldots,\lambda_n\geq 0$ and $\alpha_1,\ldots,\alpha_n\in A$ we have \eqref{eq,diag} where $f(v_{\alpha_i})$ is replaced by $w_{a_i}$, provided that $\sum_{i=1}^n \lambda_i v_{\alpha_i}\in C^\circ$. However, in general infinite dimensional order unit spaces  most elements of the interior of the cone cannot be written as a positive linear combination of finitely many positive extreme vectors and, thus, the use of this result is limited.

Let us next consider positively homogeneous order-isomorphisms. If $U\subseteq X$ and $V\subseteq Y$ are such that $\lambda u\in U$ and $\lambda v\in V$ for every $u\in U$, $v\in V$, and $\lambda>0$, then a map $f\colon U\to V$ is called \emph{positively homogeneous} if $f(\lambda u)=\lambda f(u)$ for every $u\in U$ and $\lambda>0$. If $U$ and $V$ are generating Archimedean cones, then this condition implies that $f(0)=0$, which yields the more common definition that includes $\lambda=0$. The definition given here also applies to maps  on interiors of cones.

In \cite[Theorem B]{S2}, Sch\"{a}ffer provides the next result.

\begin{theorem}[Sch\"affer]\label{thm,schomogeneous}
Let $(X,C,u)$ and $(Y,K,v)$ be order unit spaces. Then every positively homogeneous order-isomorphism $f\colon C^\circ\to K^\circ$ is linear.
\end{theorem}

The results of Section \ref{section3} yield the following alternative statement, in which the requirement of an order unit is replaced by a condition involving extreme rays.

\begin{theorem}\label{thrm,diag} Let $(X,C)$ and $(Y,K)$ be Archimedean partially ordered vector spaces such that $(X,C)$ is directed  and $C$ equals the inf-sup hull of $[0,\infty)_\mathcal{R}$. Then every positively homogeneous order-isomorphism $f\colon C\to K$ is linear.\end{theorem}
\begin{proof} 
We first show that $f$ is additive on $[0,\infty)_\mathcal{R}$. Let $s_1,\ldots,s_n$ be extreme vectors in $C$. It suffices to show that $f\left(\sum_{i=1}^n s_i\right)=\sum_{i=1}^n f(s_i)$. In order to apply Lemma \ref{kerstboom}, we combine terms of $s_i$ that lie on the same ray. Indeed, for $j=1,\ldots,m$, let $I_j\subseteq \{1,\ldots, n\}$ be disjoint with $\bigcup_{j=1}^m I_j=\{1,\ldots,n\}$ such that for every $i,k\in \{1,\ldots, n\}$ we have $s_i=\lambda s_k$ for some $\lambda\ge 0$ if and only if there exists $j\in \{1,\ldots,m\}$ with $i,k\in I_j$. Denote $r_j=\sum_{i\in I_j}s_i$ and for every $i\in I_j$ let $\lambda_i$ be such that $s_i=\lambda_i r_j$. Then $\sum_{i\in I_j} \lambda_i=1$ for $j=1,\ldots,m$. With the aid of Lemma \ref{kerstboom} and the positive homogeneity of $f$ we obtain
\begin{align*}
f\left(\sum_{i=1}^n s_i\right) &= f\left(\sum_{j=1}^m r_j\right)
= \sum_{j=1}^m f\left( r_j\right) =  \sum_{j=1}^m \sum_{i\in I_j} \lambda _i f\left( r_j\right)\\
&= \sum_{j=1}^m \sum_{i\in I_j}  f\left(\lambda _i r_j\right)= \sum_{i=1}^n f(s_i).
\end{align*}
As $f$ is positively homogeneous, it follows that $f$ is linear on $[0,\infty)_\mathcal{R}$. Due to Lemma \ref{lem,extend} we obtain that $f$ is linear on the inf-sup hull of $[0,\infty)_\mathcal{R}$, which equals $C$.
\end{proof}

If in Theorem \ref{thrm,diag} $f$ is an order-isomorphism from $X$ to $Y$ and $f$ is homogeneous instead of only positively homogeneous, then it can be shown along similar lines that $f$ is affine. 

It is useful to compare Theorem \ref{thm,schomogeneous} and Theorem \ref{thrm,diag}  and identify the differences. Let $(X,C,u)$ and $(Y,K,v)$ be order unit spaces. Suppose that $f\colon C\to K$ is a positively homogeneous order-isomorphism. Then straightforward verification yields $f(C^\circ)=K^\circ$. Hence it follows by Theorem \ref{thm,schomogeneous} that $f$ is linear on $C^\circ$. As $C$ is the inf hull of the convex set $C^\circ$, it follows from Lemma \ref{lem,extend} that $f$ is linear on $C$. Thus, any homogeneous order-isomorphism between cones of order unit spaces is linear. Theorem \ref{thrm,diag} provides a condition, alternative to having an order unit, that yields the same conclusion. For example, the space $\ell^p(\mathbb{N})$ for $1\le p\le \infty$ with coordinate-wise order satisfies the conditions of Theorem \ref{thrm,diag} but fails to have an order unit. Hence Sch\"affer's Theorem \ref{thm,schomogeneous} does not imply our Theorem \ref{thrm,diag}. 

Our third interest in this section is an intermediate result by Sch\"{a}ffer, which has a milder homogeneity condition than Theorem \ref{thm,schomogeneous}. In \cite[Corollary~A1]{S2} Sch\"{a}ffer shows for order unit spaces $(X,C,u)$ and $(Y,K,v)$, where either $(X,\|.\|_u)$ or $(Y,\|.\|_v)$ is separable and complete, that any order-isomorphism $f\colon C^\circ\to K^\circ$ is linear, provided there exists a $w\in C^\circ$ such that $f(\lambda w)=\lambda f(w)$ for all $\lambda\ge 0$. Compared to \cite[Theorem~B]{S2}, the positively homogeneous condition of $f$ is weakened to only being positively homogeneous on a ray through the interior of the cone, at the cost of one of the order unit spaces being separable and complete. In conjunction with Theorem \ref{thrm,ext} this yields the following. 

\begin{theorem}\label{thrm,4ray}
Let $(X,C,u)$ and $(Y,K,v)$ be order unit spaces, and $U\subseteq X$ and $V\subseteq Y$ be upper sets. Suppose that the inf-sup hull of $[0,\infty)_{\mathcal{R}_E}$ has a non-empty intersection with $C^\circ$, and that either $(X,\|.\|_u)$ or $(Y,\|.\|_v)$ is separable and complete. Then every order-isomorphism $f\colon U\to V$ is affine.
\end{theorem}
\begin{proof}
Firstly, we consider the case $U=C^\circ$ and $V=K^\circ$.
Let $C_E$ denote the inf-sup hull of the positive linear span of the engaged extreme rays of $C$.
By assumption there exists $x\in C_E\cap C^\circ$.
We recall that an order unit space is directed and Archimedean.
Hence, Proposition \ref{prop,ns} says that $f$ is affine on $C_E\cap C^\circ$.
As $f$ is an order-isomorphism mapping $C^\circ$ onto $K^\circ$, it is straightforward to infer that $f$ is in fact linear on $C_E\cap C^\circ$. 
In particular, $f(\lambda x)=\lambda f(x)$ for all $\lambda> 0$. 
Now \cite[Corollary~A1]{S2} yields that $f$ is linear on $C^\circ$.

Next we consider the case $U=C$ and $V=K$. Just as in the previous paragraph, there exists an $x\in C^\circ$ such that $f(\lambda x)=\lambda x$ for all $\lambda\geq 0$. We infer that $f(C^\circ)=K^\circ$. Indeed, let $y\in K$. As $x\in C^\circ$ there exists $\lambda \geq 0$ such that $\lambda x\geq f^{-1}(y)$. This yields that $\lambda f(x)=f(\lambda x)\geq y$.
Therefore, $f(x)$ is an order unit in $(Y,K)$ and hence $f(x)\in K^\circ$.
Now let $y\in C^\circ$. Then there exists $m>0$ such that $m x\leq y$.
We get $m f(x)=f(m x)\leq f(y)$.
In particular, $f(y)$ is an order unit and we conclude that $f(y)\in K^\circ$. Hence $f(C^\circ)\subseteq K^\circ$.
We remark that for all $\lambda\ge 0$ we have $f^{-1}(\lambda f(x))=\lambda x=\lambda f^{-1}(f(x))$, in other words $f^{-1}$ is positively homogeneous along the ray through $f(x)$.
Therefore, the previous steps applied to $f^{-1}$ instead of $f$ yield the converse inclusion $K^\circ\subseteq f(C^\circ)$.
By the first part of the proof we obtain that $f$ is linear on $C^\circ$.
Since $C$ is the inf hull of the convex set $C^\circ$, it follows from Lemma \ref{lem,extend} that $f$ is linear on $C$.

Suppose $a\in X$ and $b\in Y$ are such that $U=[a,\infty)$ and $V=[b,\infty)$. The order-isomorphism $\hat{f}$ defined by $\hat{f}(c)=f(c+a)-b$ maps $C$ to $K$. By the previously considered case $\hat{f}$ is linear, and hence $f$ is affine.

The general case where $U\subseteq X$ and $V\subseteq Y$ are upper sets follows by arguments similar to those made in the proof of Theorem \ref{thrm,ext}. Indeed, for every $a\in U$, $f$ is an order-isomorphism from $[a,\infty)$ to $[f(a),\infty)$, so that $f$ is affine on $[a,\infty)$ by the previous case. Then $f|_{[a,\infty)}$ extends to a unique affine map $F\colon X\to Y$, which is independent of $a\in U$, as $(X,C)$ is directed.\end{proof}

To conclude the paper we provide an example to which Theorem \ref{thrm,4ray} applies, but not Theorem \ref{thrm,ext}.  Consider the order unit space $(X,C,u)$ consisting of the real vector space $X=C([0,1]\cup [2,3])\oplus \mathbb{R}$, the Archimedean cone 
\[C=\{(f,\lambda)\colon \|f\|_\infty\leq\lambda\}\]
and the order unit $u=(0,1)\in C$. Then $(X,\|.\|_u)$ is complete and separable.  
The unit ball
\[ B =\{f\in C([0,1]\cup [2,3])\colon \|f\|_\infty\leq 1\}
\]
has four extreme points: $\pm\mathds{1}_{[0,1]}$ and $\pm\mathds{1}_{[2,3]}$, where $\mathds{1}_{[0,1]}$ and $\mathds{1}_{[2,3]}$ denote the indicator functions of $[0,1]$ and $[2,3]$, respectively. Therefore, $C$ has four extreme rays, namely the rays through  $(\pm\mathds{1}_{[0,1]},1)$ and $(\pm\mathds{1}_{[2,3]},1)$.   
As  \[(\mathds{1}_{[0,1]},1)+(-\mathds{1}_{[0,1]},1)=2u=(\mathds{1}_{[2,3]},1)+(-\mathds{1}_{[2,3]},1),\] 
all four extreme rays are engaged, and $u$ which lies in $C^\circ$ is contained in the positive linear span of the engaged extreme rays. We conclude that the order unit space $ (X,C,u)$ satisfies the 
conditions of Theorem  \ref{thrm,4ray}. However, the inf-sup hull of the sum of the engaged extreme rays consist only of elements of the form $(\lambda\mathds{1}_{[0,1]}+\mu\mathds{1}_{[2,3]},\nu)$, with $\lambda,\mu\ge 0$ and $|\lambda|,|\mu|\le\nu$, and hence $(X,C)$ does not satisfy the conditions of Theorem \ref{thrm,ext}.


\end{document}